\documentclass[12pt,a4paper,reqno]{amsart} 
\usepackage{amssymb,amsfonts,a4wide,mathptmx,hyperref}
\usepackage{epsf}

\numberwithin{equation}{section}

\newtheorem{thm}{Theorem}
\numberwithin{thm}{section}
\newtheorem{prop}[thm]{Proposition}

\newtheorem{cor}[thm]{Corollary}

\theoremstyle{definition}
\newtheorem{exmp}[thm]{Example}
 
\theoremstyle{remark}
\newtheorem{rem}[thm]{Remark}

\newcommand{\Log}[1]{\operatorname{log}_{\circ {#1}}}

\title[Nowhere-differentiability]{Simple proofs of 
nowhere-differentiability for Weierstrass's
function and cases of slow growth}
\author{Jon Johnsen}
\address{Department of Mathematical Sciences,
Aalborg University, Fredrik Bajers vej 7G, DK-9220 Aalborg {\O}st, Denmark}
\email{jjohnsen@math.aau.dk}
\keywords{nowhere-differentiability, Weierstrass function, lacunary Fourier
series, second microlocalisation
\\[5\jot] {\tt Appeared in 
J.~Fourier Anal.~Appl.~{\bf 16\/} (2010), pp.~17--33.}
}
\begin{document}
\enlargethispage{2\baselineskip}

 \begin{abstract}
Using a few basics from integration theory, a short proof of 
nowhere-differentiability of Weierstrass functions is given. 
Restated in terms of the
Fourier transformation, the method consists in principle 
of a second microlocalisation,
which is used to derive two general results on existence of nowhere
differentiable functions. Examples are given in which the frequencies 
are of polynomial growth and of
almost quadratic growth as a borderline case.
 \end{abstract}
\maketitle
\setlength{\baselineskip}{1.19\baselineskip}

\section{Introduction}
In 1872, K.~Weierstrass presented his famous example of a
nowhere differentiable function $W$ on the real line $\mathbb{R}$.
With two real parameters $b\ge a>1$, this may be written as
\begin{equation}
  W(t)=\sum_{j=0}^\infty  a^{-j}\cos(b^jt), \quad t\in \mathbb{R}.
  \label{W-eq}
\end{equation}
Weierstrass proved that $W$ is continuous at every $t_0\in \mathbb{R}$, 
but not differentiable at any $t_0\in \mathbb{R}$ if
\begin{equation}
  \frac{b}{a}>1+\frac{3\pi}{2},\quad\text{$b$ is an odd integer}.
  \label{W-cnd}
\end{equation}
Subsequently several mathematicians attempted to relax condition
\eqref{W-cnd}, but with limited luck. 
Much later G.~H.~Hardy \cite{Har16} was able to remove it:

\begin{thm}[Hardy 1916]
  \label{W-thm}
For every real number $b\ge  a>1$ the functions 
\begin{equation}
 W(t)=\sum_{j=0}^\infty a^{-j}\cos(b^jt),\qquad
 S(t)=\sum_{j=0}^\infty a^{-j}\sin(b^jt),
\end{equation}
are bounded and continuous on $\mathbb{R}$, but have 
no points of differentiability.
\end{thm}

The assumption $b\ge  a$ here is optimal for every $a>1$, 
for $W$ is in $C^1(\mathbb{R})$ whenever
$\frac{b}{a}<1$, due to uniform convergence of the derivatives.
(Strangely this was unobserved in \cite[Sect.~1.2]{Har16}, where Hardy
sought to justify the sufficient condition $b\ge a$ as being more natural
than eg \eqref{W-cnd}.)
Hardy also proved that $S'(0)=+\infty $ for 
\begin{equation}
  1<a\le b< 2a-1,   
\end{equation}
so then the graph of $S(t)$ is not rough at $t=0$ (similarly
$W'(\pi/2)=+\infty $ if in addition $b\in 4\mathbb{N}+1$). 
However, Hardy's treatment is not entirely elementary
and yet it fills ca.\ 15 pages.

It is perhaps partly for this reason that 
attempts have been made over
the years to find other examples.
These have often involved a replacement of the sine and cosine above by a 
function with a zig-zag graph, the first one due to
T.~Takagi \cite{Tak03} who
introduced 
$t\mapsto \sum_{j=0}^\infty 2^{-j}\operatorname{dist}(2^jt,\mathbb{Z})$.

However, the price is that the partial
sums are not $C^1$ for such functions, and due to the dilations
every $x\in \mathbb{R}$ is a limit $x=\lim r_N$ where
each $r_N\in\mathbb{Q}$ is a point at which the $N^{\operatorname{th}}$ partial sum has no
derivatives; whence nowhere-differentiability of the sum function   
is less startling.
Nevertheless, a fine example of this sort was given in just
13 lines by J.~McCarthy \cite{McC53}.

Somewhat surprisingly, there is an equally short proof of
nowhere-differentiability for $W$ and $S$, using a few basics of integration
theory. This is explained below in the introduction.

It is a major purpose of this paper to show that the simple
method has an easy extension to large classes of 
nowhere differentiable functions.
Thus the main part of the paper
contains two general theorems, of which at least the last should be a novelty,
and it ends with new examples with slow increase of the frequencies.

\begin{rem}   \label{W-rem}
By a well-known reasoning,
$W$ is nowhere-differentiable
since the $j^{\operatorname{th}}$ term cannot cancel the oscillations of the
previous ones: 
it is out of phase with previous terms as $b>1$
and the amplitudes decay exponentially since $\frac{1}{a}<1$;
as $b\ge a>1$ the combined effect is large enough 
(vindicated by the optimality of
$b\ge a$ noted after Theorem~\ref{W-thm}).
However, it will be shown 
in Section~\ref{slow-sect} that 
frequencies growing almost quadratically suffice 
for nowhere-differentiability.
\end{rem}

To present the ideas in a clearer way, one may consider the following
function $f_\theta$ which (in this paper) serves as a typical nowhere
differentiable function,
\begin{equation}
  f_\theta(t)=\sum_{j=0}^\infty 2^{-j\theta}e^{\operatorname{i} 2^{j}t},
  \qquad 0<\theta\le 1.
\end{equation}
It is convenient to
choose an auxiliary function $\chi\colon \mathbb{R}\to\mathbb{C}$ 
thus: the Fourier transformed function 
$\mathcal{F}\chi(\tau)=\hat \chi(\tau)=\int_\mathbb{R} e^{-\operatorname{i} t\tau}\chi(t)\,dt$ is
chosen as a $C^\infty $-function fulfilling 
\begin{equation}
 \hat \chi(1)=1,\qquad \text{$\hat \chi(\tau)=0$ for 
$\tau\notin\,]\tfrac{1}{2},2[\,$};  
  \label{hatchi-eq}
\end{equation}
for example  by setting
$\hat \chi(\tau)  =
  \exp\Big(2-\frac1{(2-\tau)(\tau-1/2)}\Big)$
for $\tau\in \,]\tfrac{1}{2},2[\,$. 

Using \eqref{hatchi-eq} it is easy to show that 
$\chi(t)=\mathcal{F}^{-1}\hat \chi(t)=\tfrac{1}{2\pi}\int_{\mathbb{R}}e^{\operatorname{i} t\tau}\hat
\chi(\tau)\,d\tau$ is continuous and that 
for each $k\in \mathbb{N}_0$ the function 
$t^k \chi(t)=\mathcal{F}^{-1}(\operatorname{i}^k\hat \chi{}^{(k)})$ is bounded 
(by $\sup|\hat \chi{}^{(k)}|$).
Therefore $\chi$ is integrable, 
ie $\chi\in L_1(\mathbb{R})$, and clearly 
$\int \chi\,dt=\hat \chi(0)=0$.

With this preparation, the function $f_\theta$ is particularly
simple to treat, using only ordinary exercises in integration theory:
First one may introduce the convolution
\begin{equation}
 2^k\chi(2^k\cdot )*f_\theta(t_0)
 =\int_{\mathbb{R}} 2^k\chi(2^kt)f_\theta(t_0-t)\,dt,
 \label{chi*f-eq}
\end{equation}
which is in $L_\infty (\mathbb{R})$ since $f_\theta\in L_\infty (\mathbb{R})$
and $\chi\in  L_1(\mathbb{R})$. Secondly this will be analysed in two different ways
in the proof of

\begin{prop}   \label{ftheta-prop}
For $0<\theta\le 1$ the function $f_\theta(t)
=\sum_{j=0}^\infty  2^{-j\theta}e^{\operatorname{i} 2^j t}$ is a continuous
$2\pi$-periodic, hence bounded function $f_{\theta}\colon 
\mathbb{R}\to\mathbb{C}$ without points of differentiability.
\end{prop}
\begin{proof}
By uniform convergence $f_\theta$ is for $\theta>0$
a continuous $2\pi$-periodic and bounded function; this follows from 
Weierstrass's majorant criterion as 
$\sum2^{-j\theta}<\infty $.

Inserting the series defining  $f_\theta$ into \eqref{chi*f-eq},
Lebesgue's theorem on majorised convergence allows the sum and integral to
be interchanged (eg with $\tfrac{2^k}{1-2^{-\theta}}|\chi(2^k t)|$ 
as a majorant), ie
\begin{equation}
  \begin{split}
  2^k\chi(2^k\cdot )*f_\theta(t_0)
  &=\lim_{N\to\infty }\sum_{j=0}^N 2^{-j\theta}
     \int_{\mathbb{R}} 2^k\chi(2^kt)e^{\operatorname{i} 2^j(t_0-t)}\,dt
\\
    &=\sum_{j=0}^\infty 2^{-j\theta}
   e^{\operatorname{i} 2^jt_0}\int_{\mathbb{R}} e^{-\operatorname{i} z2^{j-k}} \chi(z)\,dz
  =2^{-k\theta}e^{\operatorname{i} 2^k t_0}\hat \chi(1)=2^{-k\theta}e^{\operatorname{i} 2^k t_0}.
  \end{split}
  \label{conv-eq}
\end{equation}
Here it was also used that $\hat\chi(2^{j-k})=1$ for $j=k$ and
equals $0$ for $j\ne k$.

Moreover, since $f_\theta(t_0)\int_{\mathbb{R}}\chi\,dz=0$ (cf the note 
prior to the proposition) this gives
\begin{equation}
  2^{-k\theta}e^{\operatorname{i} 2^kt_0}= 2^k\chi(2^k\cdot)*f_\theta(t_0)
  =\int_{\mathbb{R}} \chi(z)(f_\theta(t_0-2^{-k}z)-f_\theta(t_0))\,dz.
  \label{fint-eq}
\end{equation}
So if $f_\theta$ were differentiable at $t_0$, 
$F(h):=\tfrac{1}{h}(f_\theta(t_0+h)-f_\theta(t_0))$ would define a function 
in $C(\mathbb{R})\cap L_\infty (\mathbb{R})$ for which $F(0)=f'(t_0)$, 
and Lebesgue's theorem, applied with
$|z\chi(z)|\sup_{\mathbb{R}}|F|$ as the majorant, would imply that
\begin{equation}
  -2^{(1-\theta)k}e^{\operatorname{i} 2^kt_0}
  = \int F(-2^{-k}z)z \chi(z)\,dz
  \xrightarrow[k\to\infty]{~} f'(t_0)\int_{\mathbb{R}}z\chi(z)\,dz
  =f'(t_0)\operatorname{i} \frac{d\hat \chi}{d\tau}(0)=0;
  \label{Fint-eq}
\end{equation}
hence that $1-\theta<0$. 
This would contradict the assumption that $\theta\le 1$. 
\end{proof}

By now this argument is of course of a classical nature,
although not well established in the literature.
Eg, recently R.~Shakarchi and E.~M.~Stein treated 
nowhere-differentiability of $f_\theta$ in 
Thm.~3.1 of Chap.~1 in their treatise \cite{ShSt03} with a method they
described thus:
``The proof of the theorem is really the story of
three methods of summing a Fourier series\dots  
partial sums\dots Cesaro summability\dots delayed means.''
However, they covered $0<\theta<1$ in a few pages
with
refinements for $\theta=1$ sketched there in Problem~5.8 
based on the Poisson summation formula. 

The present proofs are
not confined to periodic functions (cf the next section), for the theory of
lacunary 
Fourier series is replaced by the Fourier transformation 
$\mathcal{F}$ and its basic properties.

Moreover, also Hardy's theorem can be obtained in this way, with a few
modifications. 
The main point is to keep the factor $e^{\operatorname{i} 2^k t_0}$ instead of
introducing $\cos(2^k t_0)$ and $\sin(2^k t_0)$ 
that appear in $W$ and $S$, but do not a priori
stay away from $0$ as $k\to\infty $. Luckily this difficulty
(which was dealt with at length in \cite{Har16}) 
disappears with the present approach:

\begin{proof}[Proof of Theorem~\ref{W-thm}]
As $a>1$, clearly $W\in C(\mathbb{R})\cap L_\infty$.
Since $b>1$ it may in this proof be arranged that
$\hat \chi(1)=1$ and $\hat \chi(\tau)\ne0$ only for $\tfrac{1}{b}<\tau<b$.
As for $f_\theta$ this gives, by Euler's formula,
\begin{equation}
  b^k\chi(b^k\cdot )*W(t_0)
  =\sum_{j=0}^\infty a^{-j}\int_{\mathbb{R}} b^k\chi(b^kt)\tfrac{1}{2}
    (e^{\operatorname{i} b^j(t_0-t)}+e^{\operatorname{i} b^j(t-t_0)})\,dt.
\end{equation}
The term $e^{\operatorname{i} b^j(t-t_0)}$ is redundant here,
for $z:=tb^k$ yields 
$\int e^{\operatorname{i} b^jt}\chi(b^k t)b^k\,dt=\int e^{\operatorname{i} zb^{j-k}}\chi(z)\,dz
=\hat \chi(-b^{j-k})=0$, as $\hat \chi$ vanishes on $\,]-\infty ,0]$. 
So as in \eqref{conv-eq}, one has
$b^k\chi(b^k\cdot )*W(t_0)=\frac{e^{\operatorname{i} b^k t_0}}{2a^k}$.

Hence existence of $W'(t_0)$ would imply that
$\lim_k(\frac{b}{a})^ke^{\operatorname{i} b^k t_0}=0$;
cf \eqref{fint-eq}--\eqref{Fint-eq}.
This would contradict that $b\ge a$,
so $W$ is nowhere differentiable. Similarly $S(t)$ is so.
\end{proof}

It is known that nowhere-differentiability of $W$ can be derived with 
wavelets, cf \cite{Hol95}; an elementary explanation has been given in 
\cite{BaDu92}, but only for $b>a$.
In comparison the above proofs are short and
cover all cases through
``first principles'' of integration theory. 

In Section~\ref{micro-sect} a general result on
nowhere differentiable functions is given. 
Refining a dilation argument, a
further extension is found in Section~\ref{diff-sect}, including functions with
polynomial frequency growth. Borderline cases with quasi-quadratic growth
are given in Section~\ref{slow-sect}.

\begin{rem}
In the proof of
Theorem~\ref{W-thm}, Lebesgue's theorem on majorised
convergence is the most advanced part.
As this result appeared in 1908, cf \cite[p.~12]{Leb08}, it seems that
the argument above could, perhaps,
have been written down a century ago.
\end{rem}


\section{Proof by microlocalisation}   \label{micro-sect}
To emphasize \emph{why} the proofs of Proposition~\ref{ftheta-prop} and
Theorem~\ref{W-thm} work, the proof of the general 
Theorem~\ref{exp-thm} below
will use the Fourier
transformation $\mathcal{F}$ more consistently.

\bigskip

To apply $\mathcal{F}$ to non-integrable functions, 
it is convenient to use a few elements of
the distribution theory of L.~Schwartz \cite{Swz66}. 
(An introduction to this could be \cite{RiYo90}.)

Recall that
$\mathcal{F}f(\tau)=\hat
f(\tau)=\int_{\mathbb{R}}e^{-\operatorname{i} t\tau} f(t)\,dt$ defines a bijection 
$\mathcal{F}\colon \mathcal{S}(\mathbb{R})\to\mathcal{S}(\mathbb{R})$, when $\mathcal{S}(\mathbb{R})$ denotes the
Schwartz space of rapidly decreasing $C^\infty $-functions. Moreover,
$\mathcal{F}$ extends by duality to the 
space $\mathcal{S}'(\mathbb{R})$ of so-called temperate distributions, which contains
$L_p(\mathbb{R})$ for $1\le p\le \infty $. 
In particular it applies to exponential
functions $e^{\operatorname{i} bt}$, and as a basic exercise this yields 
$2\pi$ times the
Dirac measure $\delta_b$, ie the point measure at $\tau=b$, 
\begin{equation}
  \mathcal{F}(e^{\operatorname{i} b\cdot })(\tau)=2\pi \delta(\tau-b)=2\pi\delta_{b}(\tau).
  \label{Fexp-eq}
\end{equation}

This applies in a discussion of the function 
$$f(t)=\sum_{j=0}^\infty a_je^{\operatorname{i} b_j t}
$$
with general amplitudes $a_j\in\mathbb{C}$ 
and frequencies $0<b_0<b_1<\dots <b_j<\dots $ with $b_j\to\infty $, written 
$0<b_j\nearrow \infty $ for brevity.
(There could be finitely many $b_j\le 0$, but this would only contribute
with a $C^\infty $-term.)

Obviously the condition $\sum_j|a_j|<\infty $ implies $f\in C(\mathbb{R})\cap
L_\infty (\mathbb{R})$, so $f\in \mathcal{S}'(\mathbb{R})$, and since 
$\mathcal{F}$ applies termwise (it is continuous on $\mathcal{S}'(\mathbb{R})$),
one has by \eqref{Fexp-eq}
\begin{equation}
  \mathcal{F}f=\sum_{j=0}^\infty  a_j\mathcal{F}(e^{\operatorname{i} b_j\cdot } )
  = 2\pi\sum_{j=0}^\infty  a_j\delta_{b_j}.
  \label{Fftheta-eq}
\end{equation}
Of course \eqref{Fftheta-eq} just expresses 
that $f$ is synthesized from the frequencies
$b_0, b_1,\dots$

When $\liminf\tfrac{b_{j+1}}{b_j}>1$, then
each frequency may be picked out in a well-known way:
fixing $\lambda\in\,]1,\liminf\tfrac{b_{j+1}}{b_j}[\,$ there is a
$\chi\in \mathcal{S}(\mathbb{R})$ for which $\hat \chi(1)=1$
while $\hat \chi(\tau)\ne0$ only for 
$\tfrac{1}{\lambda}<\tau<\lambda$.  
Then $b_{k}>\lambda b_{k-1}$ 
for all $k\ge K$, if $K$ is chosen appropriately.

Considering only $k\ge K$ in the following,  
one has $\hat \chi(\tau/b_k)\ne0$ only for $\tau\in
\,]\tfrac{b_k}{\lambda};\lambda b_k[\,$.
Because $[\tfrac{b_k}{\lambda};\lambda b_k]\subset \,]b_{k-1};b_{k+1}[\,$
and $(b_k)$ is monotone increasing, 
\begin{equation}
 \hat \chi(\tau/b_k)\delta_{b_j}(\tau)=
\begin{cases}
  0&\text{ for $j\ne k$} \\
  \delta_{b_k} &\text{ for $j=k$.}
\end{cases}
  \label{xd-eq}
\end{equation}
In general $\mathcal{F}(\chi*f)=\hat \chi\cdot \hat f$
holds for all $\chi\in \mathcal{S}(\mathbb{R})$ and 
$f\in L_\infty \subset \mathcal{S}'(\mathbb{R})$, whence
\begin{equation}
 \mathcal{F}(b_k\chi(b_k\cdot )*f)= 
 \hat \chi(\cdot/b_k )\cdot \mathcal{F}f
 =2\pi\sum_{j=0}^\infty a_j\hat \chi(\cdot/b_k )\delta_{b_j}
 =2\pi a_k\delta_{b_k}.
  \label{akdelta-eq}
\end{equation}
So by use of $\mathcal{F}^{-1}$ and \eqref{Fexp-eq},
\begin{equation}
  b_k\chi(b_k\cdot )*f(t)= 2\pi a_k\mathcal{F}^{-1}\delta_{b_k}(t)
  = a_ke^{\operatorname{i} b_k t}.
  \label{conv'-eq}
\end{equation}
This gives back \eqref{conv-eq} in case $a_k=2^{-k\theta}$ and $b_k=2^k$, 
but the derivation above 
is more transparent than eg the proof 
of \eqref{conv-eq}, since it is clear why
convolution by $b_k\chi(b_k\cdot )$ just gives the $k^{\operatorname{th}}$ term.

The process in \eqref{akdelta-eq}--\eqref{conv'-eq} 
has of course been known for ages, but with 
distribution theory it is fully justified although $\mathcal{F}f$ consists
of measures. In principle, it is a banal example of  
what is sometimes called a \emph{second microlocalisation} of $f$, since 
$\hat \chi(b_k\tau)\mathcal{F}f(\tau)$ 
is localised to frequencies $\tau$ restricted in \emph{both} size and direction;
namely to $|\tau|\approx b_k$ and $\tau>0$, respectively.

The second microlocalisation is more visible in a separate treatment of
\begin{equation}
  \operatorname{Re} f(t)=\sum_{j=0}^\infty  a_j\cos(b_j t),
\qquad
  \operatorname{Im} f(t)=\sum_{j=0}^\infty  a_j\sin(b_j t).
\end{equation}
Indeed, by Euler's formula and \eqref{Fexp-eq},
\begin{equation}
  \mathcal{F}\cos(b_j\cdot ) 
  =\tfrac{2\pi}{2}(\delta_{b_j}+\delta_{-b_j}),\qquad
  \mathcal{F}\sin(b_j\cdot )= \tfrac{2\pi}{2\operatorname{i}}(\delta_{b_j}-\delta_{-b_j}).
  \label{Fsincos-eq}
\end{equation}
Here multiplication by $\hat \chi(\cdot/b_j )$ \emph{removes} 
the contribution from $\delta_{-b_j}$ since $\hat \chi$ vanishes on
$\,]-\infty ,0]$.  
This actually explains why the proof of Theorem~\ref{W-thm} was saved by the
redundancy of the term $e^{\operatorname{i} b^j(t-t_0)}$.

However, the details will follow in connection with the next result.
Recall that $f\colon \mathbb{R}\to\mathbb{C}$ is said to be Lipschitz
continuous at $t_0$ if there exist two constants 
$L>0$, $\eta>0$ such that
$|f(t)-f(t_0)|\le L|t-t_0|$ for every $t\in \,]t_0-\eta,t_0+\eta[\,$. 

\begin{thm}
  \label{exp-thm}
Let $f\colon \mathbb{R}\to \mathbb{C}$ be given as 
$f(t)=\sum_{j=0}^\infty a_j \exp({\operatorname{i} b_jt})$
for a complex sequence $(a_j)_{j\in\mathbb{N}_0}$ with $\sum_{j=0}^\infty |a_j|
<\infty$ 
and $   0<b_j\nearrow \infty$ satisfying
\begin{equation}
  \liminf_{j\to\infty } \frac{b_{j+1}}{b_j}>1,\qquad
  a_jb_j\not\to 0\quad\text{for}\quad j\to\infty.
  \label{ab-cnd}
\end{equation}
Then $f$ is bounded and continuous on $\mathbb{R}$, but
nowhere differentiable. If $\sup_{j}|a_j|b_j=\infty $ holds in addition,
then $f$ is not Lipschitz continuous at any point.
The conclusions are also valid
for $\operatorname{Re} f$ and $\operatorname{Im} f$.
\end{thm}

\begin{proof} Continuing from \eqref{conv'-eq}, one clearly has
$a_ke^{\operatorname{i} b_kt_0}=\int_{\mathbb{R}}\chi(z)f(t_0-z/b_k)\,dz$.

If $f$ were differentiable at $t_0$, then
$F(t)=(f(t_0+t)-f(t_0))/t$ would be in $L_\infty $ 
(like $f$), so since $\int \chi(t)\,dt=0$,  
multiplication by $b_k$ and majorised convergence would imply
\begin{equation}
  -a_kb_ke^{\operatorname{i} b_k t_0}
  =\int_{\mathbb{R}} z\chi(z)\frac{f(t_0-z/b_k)-f(t_0)}{-z/b_k}\,dz
 \xrightarrow[k\to\infty ]{~} 
  f'(t_0) \operatorname{i} \frac{d\hat \chi}{d\tau}(0)=0.
  \label{akbk-eq}
\end{equation}
This would entail $|a_k|b_k\to 0$ for
$k\to\infty $, in contradiction of \eqref{ab-cnd}. 

In addition, were $f$ Lipschitz continuous at $t_0$, 
then again $F$ would be bounded, so the integral in 
\eqref{akbk-eq} would be uniformly bounded 
with respect to $k$, 
in which case $\sup_k |a_k|b_k<\infty$.

Finally, using \eqref{Fsincos-eq} ff, 
one can clearly replace $f$ in \eqref{akdelta-eq}--\eqref{conv'-eq} 
by $\operatorname{Re} f$ or $\operatorname{Im} f$ if only $a_k$ is replaced by
$a_k/2$ and $a_k/(2\operatorname{i})$, respectively. Eg
\begin{equation}
  \hat \chi(\tau /b_k)\mathcal{F}\operatorname{Im} f(\tau)=
  2\pi\sum_{j=0}^\infty \hat \chi(\tau/b_k)
     \frac{a_j}{2\operatorname{i}}(\delta_{b_j}(\tau)-\delta_{-b_j}(\tau))
  =2\pi\frac{a_k}{2\operatorname{i}}\delta_{b_k}(\tau).
\end{equation}
Proceeding as for $f$ itself via variants of \eqref{conv'-eq} 
and \eqref{akbk-eq}, it
follows that neither $\operatorname{Re} f$ nor $\operatorname{Im} f$ can be differentiable at some
$t_0\in \mathbb{R}$, respectively Lipschitz continuous if 
$\sup_j |a_j|b_j=\infty $.
\end{proof}

Clearly $\limsup |a_j|b_j>0$ is equivalent to $a_jb_j\not\to 0$; cf
\eqref{ab-cnd}.  
While the former leaves a gap to the non-Lipschitz condition, 
the latter is natural as termwise
differentiation yields $\sum a_jb_j e^{\operatorname{i} b_j t}$, which cannot converge
unless $a_jb_j\to 0$.
The conditions \eqref{ab-cnd} have been used repeatedly in the literature, 
but Theorem~\ref{exp-thm} should be of interest
because of the easy treatment of non-periodic $f$ as well as of
$\operatorname{Re} f$, $\operatorname{Im} f$.

\begin{rem}   \label{Hoelder-rem}
A necessary condition for H{\"o}lder continuity of order $\alpha\in \,]0,1[\,$
follows at once from a modification of the above argument: replacing $a_kb_k$
on the left-hand side of \eqref{akbk-eq} by $a_k b_k^\alpha$, 
the resulting integral will be
uniformly bounded with respect to $k$ since $\int
|z|^\alpha|\chi(z)|\,dz<\infty $. Hence
\begin{equation}
  \sup_{k} |a_k|b_k^ \alpha<\infty 
\end{equation}
whenever $f(t)$ in Theorem~\ref{exp-thm} is H{\"o}lder
continuous of order $\alpha$ at a single point $t_0$. 
\end{rem}

\begin{exmp}
  \label{W-exmp}
Sequences of power type like $a_j=a^{-j}$ and $b_j=b^j$ for parameters $b\ge
|a|>1$ give $f(t)=\sum_{j=0}^{\infty }a^{-j}e^{\operatorname{i} b^jt}$, 
which is covered by Theorem~\ref{exp-thm} as 
$|a_j|b_j=|\tfrac{b}{a}|^j\ge 1$ and $\tfrac{b_{j+1}}{b_j}=
b>1$. Therefore Theorem~\ref{exp-thm} contains Theorem~\ref{W-thm} and
extends it to complex amplitudes.

For $W(t)$ Remark~\ref{Hoelder-rem} yields
$\tfrac{b^\alpha}{a}\le 1$, ie $\alpha\le \tfrac{\log a}{\log b}$. 
In case $b>a>1$, Hardy's proof strategy \cite[p.~311]{Har16} was to show
that $W$ is 
H{\"o}lder continuous
of order $\alpha=\log a/\log b$ but no better (even locally); 
whereas for $b=a>1$
it was obtained that $W(t+h)-W(t)={\mathcal O}(|h|\log1/|h|)$.
So Remark~\ref{Hoelder-rem} at once gives a sharp upper bound for the
H{\"o}lder exponent of $W$. (This was mentioned as a difficult task in
\cite{Jaf97}; however, 
\cite{BaDu92} contains a relatively short proof of the bound.)
Thm.~4.9 in Ch.~II of Zygmund's book
\cite{Zygm} also treats H{\"o}lder continuity of $W$.
\end{exmp}

\begin{exmp}
  \label{Darb-exmp}
In the same way Theorem~\ref{exp-thm} covers Darboux's function  
$f(t)=\sum_{j=0}^\infty \frac{\sin((j+1)!t)}{j!}$,
for $a_j=1/j!$ and $b_j=(j+1)!$ fulfil in particular 
$\tfrac{b_{j+1}}{b_j}=j+2\nearrow \infty $ and $a_jb_j=j+1\nearrow\infty $.
\end{exmp}

\begin{exmp}
  \label{bj-exmp}
Setting  $a_j=a^{-j}$ for some $a>1$
and defining $(b_j)$ by $b_{2m}=a^{2m}$
and $b_{2m+1}=(1+a^{-p})a^{2m}$, it is seen directly that
when the power $p$ is so large that $1+a^{-p}<a^2$, 
then the sequences $(a_j)$ and
$(b_j)$ fulfil the conditions of Theorem~\ref{exp-thm}. Eg
\eqref{ab-cnd} holds as
$\tfrac{b_{j+1}}{b_j}\in \{1+a^{-p},a^2(1+a^{-p})^{-1}\}\subset \,]1,\infty
[\,$
and  $a_jb_j\in \{1,(1+a^{-p})/a\}$. Thus
$f(t)$ is nowhere differentiable in this case. If further $p$ is so large
that $1+a^{2}<a^p(a^{2}-1)$ it is easily verified that $b_{2m+1}-b_{2m}<
b_{2m}-b_{2m-1}$ so that $(b_{j+1}-b_j)$ is not monotone increasing.
Eg if $a=5$,  both requirements are met by $p=1$ and the values of $(b_j)$
are 
$$1, \tfrac{6}{5}, 25, 30, 625, 750, 15\,625, 18\,750,\dots $$ 
Clearly these frequencies
have a distribution with lacunas of rather uneven size.
\end{exmp}

To elucidate the assumptions in Theorem~\ref{exp-thm}, note that
for a sequence $(b_j)$ of positive reals,
\begin{equation}
  \liminf \frac{b_{j+1}}{b_j}> 1
\iff   \liminf \frac{b_{j+1}-b_j}{b_j}>0
\iff  \exists J,\varepsilon>0\ 
\forall j>J\colon
\varepsilon b_j<b_{j+1}-b_j < b_{j+1}.
  \label{bDbb-eq}
\end{equation}
Hence the conditions $\liminf \tfrac{b_{j+1}}{b_j}>1$ 
and $b_j\nearrow \infty $ imply that
$b_{j+1}-b_j\to\infty $ 
when Theorem~\ref{exp-thm} applies; but the gaps $b_{j+1}-b_j$ need not be
monotone increasing; cf Example~\ref{bj-exmp}.

Moreover, $\liminf \tfrac{b_{j+1}}{b_j}>1$ implies exponential
growth of the $b_j$ (as $b_j\ge \lambda^{j-j_0}b_{j_0}$ when
$\tfrac{b_{j+1}}{b_j}\ge \lambda>1$ for $j\ge j_0$)
so Thm.~\ref{exp-thm} does not apply if $b_j=j^q$.
This will be remedied in Thm.~\ref{diff-thm} ff.

\section{Dilation by differences}  \label{diff-sect}
To escape the exponential frequency growth in Theorem~\ref{exp-thm}, it
is natural instead of dilation by $b_j$ 
to use the smallest gap at frequency $b_j$, ie to dilate by
\begin{equation}
  \Delta b_j=\min(b_j-b_{j-1},b_{j+1}-b_j). \qquad (b_{-1}=0)
  \label{Dbj-eq}
\end{equation}
This requires $\lim\Delta b_j=\infty $,  
that one could use as an assumption
(replacing exponential growth by one of its consequences,
cf \eqref{bDbb-eq}).
However, 
\eqref{aDb-eq} below is weaker, since it only implies the existence of
$j_1<j_2<\dots $  satisfying
$\lim \Delta b_{j_k}=\infty $
(the $\Delta b_j$ are unbounded since $a_j\to0 $).

\begin{thm}
  \label{diff-thm}
Let 
$f(t)=\sum_{j=0}^\infty  a_j\exp({\operatorname{i} b_jt})$
for a complex sequence $(a_j)_{j\in \mathbb{N}_0}$ with
$\sum_{j=0}^\infty  |a_j|<\infty $
and $0<b_j\nearrow \infty $.
When $\Delta b_j$ in \eqref{Dbj-eq}
fulfils 
\begin{equation}
  a_j\Delta b_j \not\to 0 \quad\text{for $j\to\infty $},
  \label{aDb-eq}
\end{equation}
then $f$ is bounded and continuous on $\mathbb{R}$, but nowhere
differentiable.  If $\sup_j |a_j|\Delta b_j=\infty $ holds in addition, 
then $f$ is 
not Lipschitz continuous at any $t_0\in \mathbb{R}$. The
conclusions are also valid for $\operatorname{Re} f$ and $\operatorname{Im} f$. 
\end{thm}

\begin{proof}
That $f\in  C(\mathbb{R})\cap  L_\infty (\mathbb{R})$ is shown as in Theorem~\ref{exp-thm}.
Let now $\mathcal{F}\psi\in  C^\infty (\mathbb{R})$ fulfil $\mathcal{F}\psi(0)=1$ and
$\mathcal{F} \psi(\tau)\ne0$ only for $|\tau|<1/2$, 
and take the spectral cut-off function as
\begin{equation}
  \hat\psi_k(\tau)= \hat \psi(\frac{\tau-b_k}{\Delta b_k}).
\end{equation}
Then the definition of $\Delta b_k$ as a minimum entails
\begin{equation}
  \hat \psi_k(\tau)\ne0\implies
  b_k-\tfrac{1}{2}(b_k-b_{k-1})<\tau<b_k+\tfrac{1}{2}(b_{k+1}-b_k).
  \label{psik0-eq}
\end{equation}
Since $(b_j)$ is increasing, the $\tau$-interval specified here 
only contains $b_j$ for $j=k$, whence
\begin{equation}
  \hat \psi_k(\tau)\hat f(\tau)=
 2\pi\sum_{j=0}^\infty  a_j\hat \psi_k(\tau)\delta_{b_j}(\tau)
  =2\pi a_k\delta_{b_k}(\tau).
  \label{psikf-eq}
\end{equation}
Note that by a change of variables,
\begin{equation}
  \psi_k(t)=
  \mathcal{F}^{-1}\hat \psi_k(t)=\tfrac{1}{2\pi} \int_{\mathbb{R}}
  e^{\operatorname{i} t(b_k+\sigma\Delta b_k)}\hat \psi(\sigma)\Delta b_k\,d\sigma
  = (\Delta b_k)e^{\operatorname{i} tb_k}\psi(t\Delta b_k).
\end{equation}
Here the integral of the left-hand side is $0$ by \eqref{psik0-eq}, 
so application of $\mathcal{F}^{-1}$ to \eqref{psikf-eq} gives,
\begin{equation}
  \begin{split}
    a_k(\Delta b_k)e^{\operatorname{i} b_k t_0}&= (\Delta b_k) f* \psi_k(t_0)
\\
  &=\int_{\mathbb{R}}(f(t_0-t)-f(t_0))(\Delta b_k)^2 e^{\operatorname{i} b_kt}\psi(t\Delta b_k)
\, dt.
\\
  &=\int_{\mathbb{R}}\frac{f(t_0-z/\Delta b_k)-f(t_0)}{z/\Delta b_k}
     z\psi(z) e^{\operatorname{i} z{b_k}/{\Delta b_k}}\, dz.
  \end{split}
  \label{diff-id}
\end{equation}
If $f$ is Lipschitz continuous at $t_0$, $h\mapsto
(f(t_0+h)-f(t_0))/h$ is bounded, so for some $L\in \mathbb{R}$
\begin{equation}
 \sup_k |a_k|\Delta b_k 
  \le  \sup_k\int_{\mathbb{R}} \left|
  \frac{f(t_0-z/\Delta b_k)-f(t_0)}{z/\Delta b_k} \right||z\psi(z)|\,dz
\le  L\int_{\mathbb{R}} |z\psi(z)|\,dz<\infty .
\end{equation}
Moreover, because
$b_k/\Delta b_k\ge  b_k/(b_k-b_{k-1})>1$, 
\begin{equation}
  \int_{\mathbb{R}} z\psi(z) 
  e^{\operatorname{i} z b_k/\Delta b_k}\,dz
  =    \operatorname{i} \frac{d\hat \psi}{d\tau}(-{b_k}/{\Delta b_k})=0.
\end{equation}
So were $f$ differentiable at $t_0$, it would follow from 
\eqref{diff-id} by majorised convergence that 
\begin{equation}
  a_k(\Delta b_k) e^{\operatorname{i} t_0 b_k}
  = -\int_{\mathbb{R}}
   \Big(\frac{f(t_0-z/\Delta b_k)-f(t_0)}{-z/\Delta b_k}-f'(t_0)\Big)
  z\psi(z) e^{\operatorname{i} z{b_k}/{\Delta b_k}}\,dz 
  \xrightarrow[k\to\infty ]{~}0,
\end{equation}
in contradiction of \eqref{aDb-eq}.
Finally the same arguments apply to $\operatorname{Re} f$, $\operatorname{Im} f$ by dividing  $a_k$ by
$2$ and $2\operatorname{i}$, respectively, as in Theorem~\ref{exp-thm}.
\end{proof}

\begin{rem}
  \label{Hoelder'-rem}
 If $f$ in
Theorem~\ref{diff-thm} is H{\"o}lder continuous of order $\alpha\in
\,]0,1[\,$ at some $t_0$, \eqref{diff-id} yields
\begin{equation}
  \sup_{j} |a_j|(\Delta b_j)^\alpha<\infty .
  \label{ggc-eq}
\end{equation}
When applied to $W$, this gives the same result as
Remark~\ref{Hoelder-rem}, for $\Delta b_j=cb^j$ with $c=1-1/b>0$ as $b>1$.
Hence the gap growth condition \eqref{ggc-eq} cannot be sharpened in
general.
\end{rem}
Note that, due to the use of the extended $\mathcal{F}$ on $\mathcal{S}'(\mathbb{R})$, 
it is clear from 
\eqref{psikf-eq} that one cannot dilate by larger quantities than $\Delta
b_k$, so the method seems optimally exploited.

Apparently, nowhere-differentiability has not been obtained under
the weak assumptions of Theorem~\ref{diff-thm} before.
Like for $f_\theta$ and
$W$, the regularity of the sum
function improves when the growth of the frequencies is taken smaller,
eg by reducing $q$ in the following:

\begin{exmp}[Polynomial growth]  \label{pwr-exmp}
For $p>1$ one has uniformly continuous functions 
\begin{equation}
  f_{p,q}(t)=\sum_{j=1}^\infty \frac{\exp(\operatorname{i} t j^q)}{j^p},
   \qquad \operatorname{Re} f_{p,q}(t),\qquad \operatorname{Im} f_{p,q}(t),
  \label{fpq-eq}
\end{equation}
that moreover are $C^1$ 
and bounded with bounded derivatives on $\mathbb{R}$ in case $0<q<p-1$.
However, for $q\ge  p+1$ they are 
nowhere differentiable according to Theorem~\ref{diff-thm}: 
\eqref{aDb-eq} follows since 
by the mean value theorem the frequency gaps increase, and
\begin{equation}
  \limsup j^{-p}(j^q-(j-1)^q)\ge \limsup qj^{q-p-1}(1-1/j)^{q-1}
=\begin{cases}
  q &\text{ for $q=p+1$,}
\\
  \infty &\text{ for $q>p+1$}.
\end{cases}
\end{equation}
Moreover, for $q>p+1$ there is not Lipschitz continuity  at any point.

But the functions in \eqref{fpq-eq} are globally 
H{\"o}lder continuous of order $\alpha=(p-1)/q$
if only $q> p-1$. This results from integral comparisons that 
(eg for $c=1+1/(q-p+1)$) yield
\begin{equation}
  |f_{p,q}(t+h)-f_{p,q}(t)|\le \sum_{j\le  N} j^{q-p}|h| +\sum_{j>N} 2j^{-p}
  \le c N^{q-p+1}|h| +\frac{2}{p-1}N^{1-p}.
  \label{Nsplit-eq}
\end{equation}
For $0<|h|\le \tfrac{1}{2}$ this is exploited for the unique $N$ such that
$N\le |h|^{-1/q}<N+1$.
For the H{\"o}lder exponents, 
this is optimal among the powers $|h|^{-\theta}$, 
for clearly $\theta=1/q$ maximises
\begin{equation}
  \min(\theta(p-1),1-\theta(q-p+1)).
\end{equation}
Insertion of the choice of $N$ in \eqref{Nsplit-eq}
gives a $C<\infty $ so that for $h\in \mathbb{R}$ (as $f_{p,q}\in L_\infty$)
\begin{equation}
  |f_{p,q}(t+h)-f_{p,q}(t)|\le C|h|^{\alpha}, \qquad \alpha=\frac{p-1}{q}.
  \label{fpqa-eq}
\end{equation}

Since $\Delta b_j<qj^{q-1}$ for $q>1$, 
the necessary condition in Remark~\ref{Hoelder'-rem} is fulfilled for
$\alpha(q-1)-p\le 0$, leading to the upper bound
$\alpha\le \frac{p}{q-1}$. So in view of \eqref{fpqa-eq} 
there remains a gap for these functions.
\end{exmp}

In view of Example~\ref{pwr-exmp}, it is clear that 
Theorem~\ref{diff-thm} improves
Theorem~\ref{exp-thm} a good deal. The condition $|a_j|\Delta
b_j\not\to0$ in \eqref{aDb-eq}
cannot be relaxed in general, for already for $W$ it amounts
to $b\ge a$, that is equivalent to nowhere-differentiability.

However, \eqref{aDb-eq} does not give optimal results for $f_{p,q}$.
Eg the case with $p=q=2$ has been completely clarified and 
shown to have a delicate nature, as it is known from several investigations
that the so-called Riemann function  
\begin{equation}
 R(t)=\sum_{j=1}^{\infty }\frac{\sin(\pi j^2t)}{j^2}  
\end{equation}
is differentiable with $R'(t)=-1/2$
exactly at $t=r/s$ for odd integers $r$, $s$. 
For properties of this function the reader is referred to the paper of 
J.~Duistermaat \cite{Dui91}.

As $f_{p,q}$ is in $C^{1}(\mathbb{R})$ for every
$q<p-1$ when $p>1$, transition to
nowhere-differentiability occurs (perhaps gradually) as $q$ runs through
the interval $[p-1,p+1[\,$. 
Nowhere-differentiability for $q\ge p+1 $ was also mentioned for $\operatorname{Im} f_{p,q}$ 
by W.~Luther \cite{Lut86}
as an outcome of a very general Tauberian theorem.
(In addition $\operatorname{Im} f_{p,2}$ was
covered with nowhere-differentiability for $p\le 3/2$ providing cases in
$[p+\tfrac{1}{2},p+1[\,$; for $t$ irrational
Luther's result relied on Hardy's investigation \cite{Har16}, that covered 
$\operatorname{Im} f_{p,2}$ for $p<5/2$ thus giving cases of almost nowhere differentiability
in $\,]p-\tfrac{1}{2},p+\tfrac{1}{2}]$ for $q=2$.)

By \eqref{fpqa-eq}, $R$ is
globally H{\"o}lder continuous of order $\alpha=1/2$, that is well known 
\cite{Dui91}. At the differentiability points, this is of course not
optimal, but the local H{\"o}lder regularity of $R$ is known to attain every
value 
$\alpha\in [\tfrac{1}{2},\tfrac{3}{4}]$ in a non-empty set; cf the paper of
S.~Jaffard \cite{Jaf97}. 

\begin{rem}   \label{ChUb-rem}
Recently $f_{p,q}$ was studied by
F.~Chamizo and A.~Ubis \cite{ChUb07} for $q\in \mathbb{N}$, $p>1$,
with nowhere-differentiability 
treated by convolving $f_{p,q}$ with the Fej\'er kernel,
cf \cite[Prop.~3.3]{ChUb07}. 
This method was proposed as
an alternative to those of \cite{Lut86}
and similar in spirit to the above proof of 
Theorem~\ref{diff-thm}.
But some statements are flawed: eg in \cite[Thm.~3.1]{ChUb07}, 
$f_{p,q}$ is claimed differentiable at an irreducible fraction
$t=r/s\in \mathbb{Q}$, $s>0$, if and only if both $q<p+1/2$ and,
for some maximal prime power $\sigma^\gamma$ in the factorisation of $s$, 
$q$ divides $\gamma-1$ but is relatively prime with $\sigma-1$.
However,
$f_{p,q}\in C^1(\mathbb{R})$ for every $q<p-1$ 
(cf Example~\ref{pwr-exmp} above), 
whilst for $q\in \mathbb{N}\cap [2,p-1[\,$ their condition 
is violated at $\tfrac{r}{s}=\tfrac{1}{2^q}$;
hence this claim is not correct for such $q$.
\end{rem}


\begin{figure}[phtb]
\epsfxsize=9cm
\epsfbox{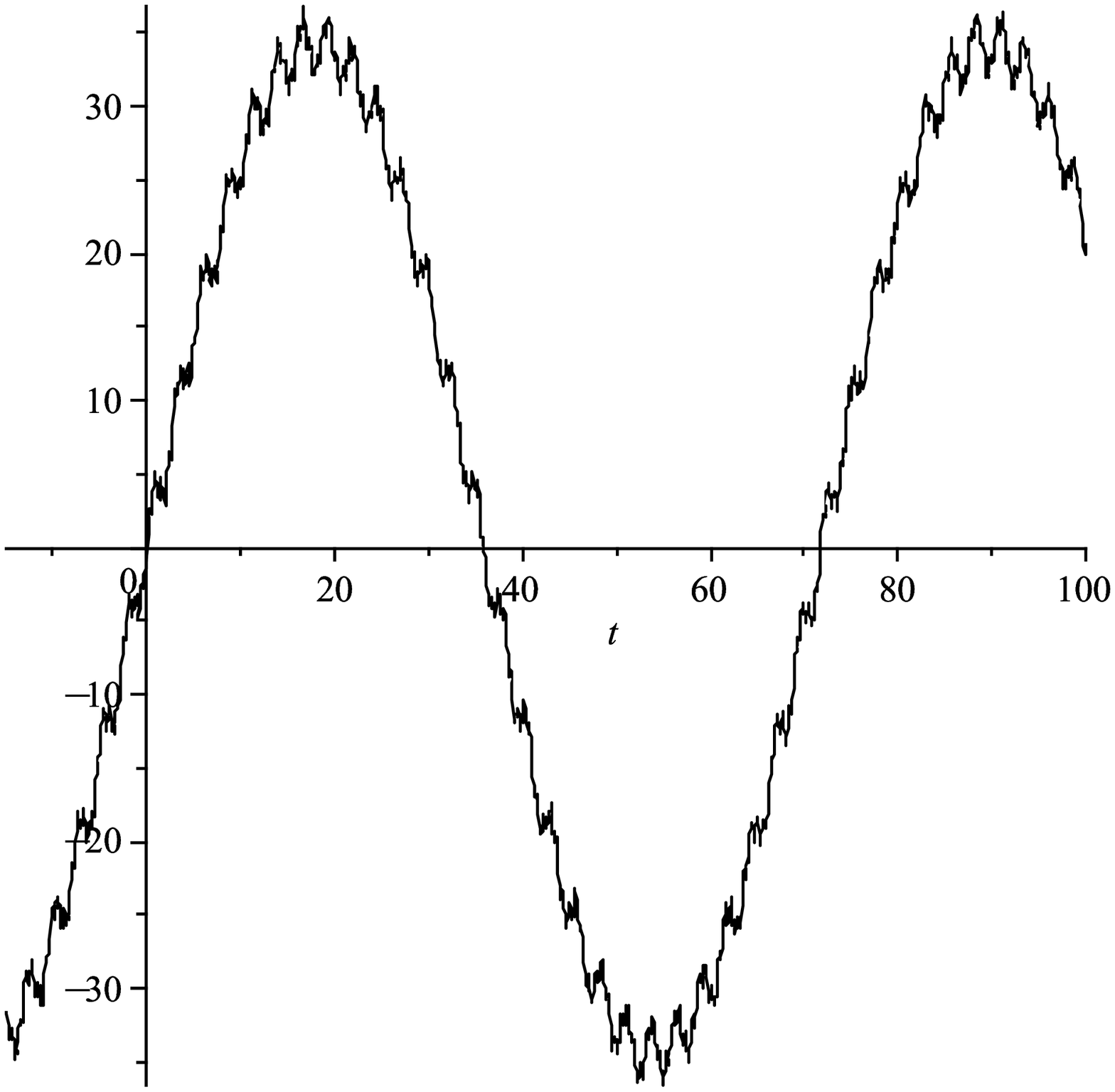}
\caption{Graph of $\operatorname{Im} F_2(t)$ for the function in \eqref{F2-eq}}
  \label{logg-fig}
\end{figure}

\section{Slowly growing frequencies}   \label{slow-sect}
Using Theorem~\ref{diff-thm},
nowhere-differentiability is obtained in new cases where 
$b_j$ is almost $\mathcal O(j^2)$.
The examples  to follow all relate to the limiting case $p=1$, $q=2$ in 
Example~\ref{pwr-exmp}.

Setting $\log^a t=(\log t)^a$ for $a\in \mathbb{R}$ and $t>1$, the functions
\begin{equation}
  F_1(t)=\sum_{j=2}^\infty \frac{\exp(\operatorname{i} tj^2\log^bj)}{j\log^aj},
  \qquad \operatorname{Re} F_1(t), \qquad \operatorname{Im} F_1(t)
\end{equation}
are for $b\ge a>1$ continuous, bounded and nowhere differentiable on $\mathbb{R}$
by  Theorem~\ref{diff-thm}, for the mean-value theorem gives
$a_j\Delta b_j\ge 2\log^{b-a}j$.
For $b>a>1$ there is no Lipschitz continuity.
 
To simplify, the Lipschitz aspect is left out below by taking
$b=a$. Instead iterated logarithms will be seen to allow
quasi-quadratic growth of the $b_j$,
relying on a general result for  $b_j=j/|a_j|$:

\begin{cor}   \label{conv-corr}
If $\sum_{j>J} |a_j|<\infty $ and $|a_j|\ge |a_{j+1}|>0$ 
for all $j>J$
while for a convex function
$\varphi\colon \,]J,\infty [\,\to\mathbb{R}$  one has $\varphi(j)=j/|a_j|$ for 
$j\in \mathbb{N}\cap\,]J,\infty [\,$, then 
\begin{equation}
  f(t)=\sum_{j>J} a_j\exp(\operatorname{i} tj/|a_j|),\qquad \operatorname{Re} f(t), \qquad \operatorname{Im} f(t)
\end{equation}
are continuous on $\mathbb{R}$ but nowhere differentiable.
\end{cor}

\begin{proof}
As $\varphi$ is convex, clearly
$\Delta b_j=\varphi(j)-\varphi(j-1)$.
Therefore
$|a_j|\Delta b_j =|a_j|(\tfrac{j}{|a_j|}-\tfrac{j-1}{|a_{j-1}|})
\ge  j-(j-1)=1$, 
and $b_j=j/|a_j|\nearrow\infty $, whence
Theorem~\ref{diff-thm} yields the claim.
\end{proof}

For $t>e$, there is a
nowhere differentiable function given by
\begin{equation}
  F_2(t)=\sum_{j=3}^\infty 
       \frac{\exp(\operatorname{i} t j^2\log j(\log\log j)^a)}{j\log j(\log\log j)^a},
  \qquad a>1.
  \label{F2-eq}
\end{equation}
This can be seen directly from Corollary~\ref{conv-corr}, but it 
is a special case of Example~\ref{logn-exmp} below.

The graph of 
$\operatorname{Im} F_2$ is sketched in
Figure~\ref{logg-fig}. All figures give a plot of a	
partial sum with 1000 terms and partition points.
The quasi-periodic behaviour visible in Figure~\ref{logg-fig}
results because the first term of the series is dominating.
More pronounced cases of slow growth are given in:

\begin{exmp}   \label{logn-exmp}
Denoting the $n$-fold logarithm by $\Log{n} t:= \log\dots \log t$, defined for
$t>E_{n-2}:=\exp\dots \exp 1$ ($n-2$ times), and setting 
$\Log{n}^at=(\Log{n} t)^a$ for $a\in \mathbb{R}$ and $t>E_{n-1}$, there is a continuous
nowhere differentiable function given for $t>E_{n-1}$ by
\begin{equation}
  F_n(t)=\sum_{j>E_{n-1}} 
       \frac{\exp(\operatorname{i} t j^2\log j\dots \Log{(n-1)}j \cdot \Log{n}^aj)}
       {j\log j\dots \Log{(n-1)}j \cdot \Log{n}^aj},\qquad a>1.
  \label{flogn-eq}
\end{equation}
Indeed, $\sum |a_j|<\infty $ because $a_j=1/(j\log j\dots\Log{(n-1)}j
\Log{n}^aj)$ equals $g'(j)$, whereby
\begin{equation}
  g(t)=\tfrac{1}{1-a}\Log{n}^{1-a}t=\tfrac{1}{1-a}(\log\dots \log t)^{1-a}
  \xrightarrow[t\to\infty ]{~} 0 
  \quad\text{for}\quad a>1.
\end{equation}
That $a_j\ge a_{j+1}$ follows since all iterated
logarithms are monotone increasing and positive for $j>E_{n-1}$.
Analogously,
$\varphi_{a,n}(t)=t^2\log t\dots \Log{(n-1)}t\cdot \Log{n}^{a}t$
is convex on $\,]E_{n-1},\infty [\,$, 
for $\varphi'_{a,n}(t)$ is easily written as 
a sum of $n+1$ terms, that are increasing.
Hence Corollary~\ref{conv-corr} gives the claim.

$\operatorname{Im} F_3$ and its first term are sketched in Figure~\ref{logn3-fig}
for $a=2$. One has $E_{2}=e^{e}\approx 15.15$, and for 
$j\ge 16$ the frequencies are 
0.28, 1.4, 3.5, 6.8, 11, 17, 25, 34, 45, 57, \dots
(As $a=2$, the sum could include $e< j<e^e$, but
the $b_j$ decrease from $11$
for $j=6$ to $0.009$ for $j=15$.)

As a last comparison, for $a=2$ and $n=4$, summation begins 
in \eqref{flogn-eq} after
$E_3=e^{e^e}=3\,814\,279.1\dots $ Cf Figure~\ref{logn4-fig}.
The quasi-quadratic growth of $b_j$ is indicated by the fact that terms
no.~1, 10, 100, 1000 have frequencies 0.02, 2.4, 247, 24\,326, respectively.
\end{exmp}

\begin{figure}[phtb]
\epsfxsize=9cm
\epsfbox{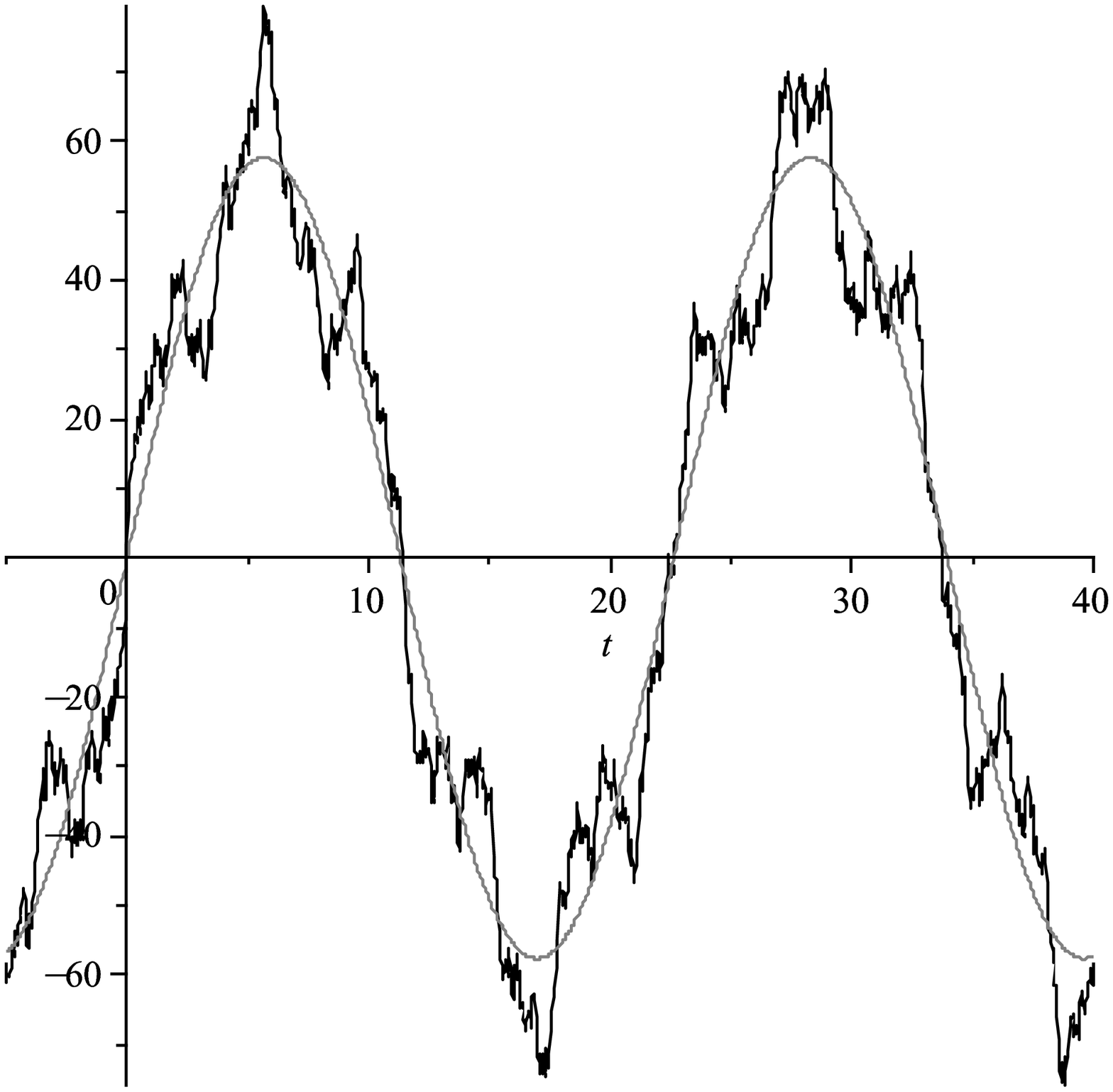}
\caption{$\operatorname{Im} F_3(t)$ 
for the function in Example~\ref{logn-exmp}, $n=3$, $a=2$}
  \label{logn3-fig}
\end{figure}
\begin{figure}[htbp]
\epsfxsize=10cm
\epsfbox{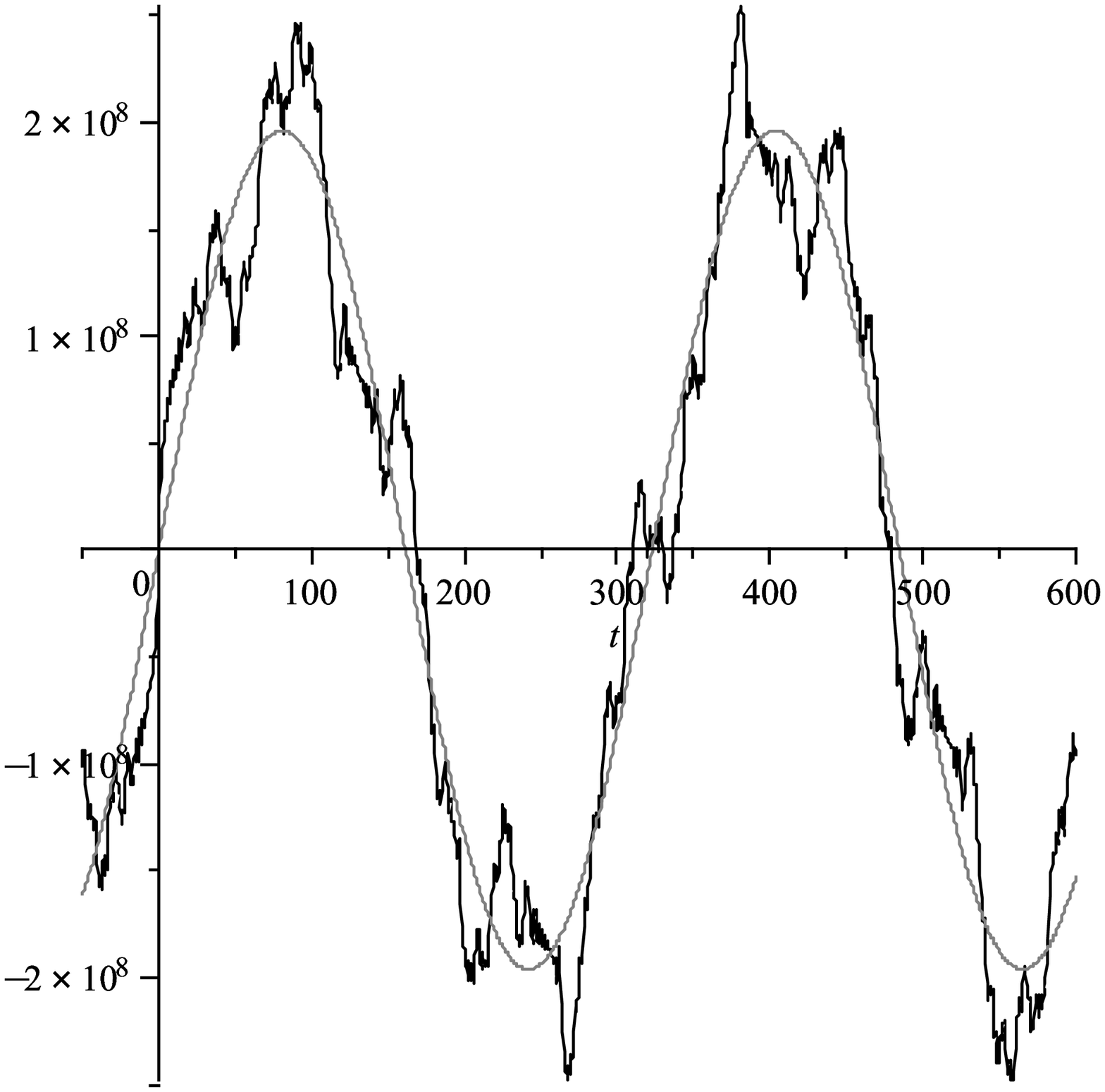}
\caption{$\operatorname{Im} F_4(t)$ 
for the function in Example~\ref{logn-exmp}, $n=4$, $a=2$}
  \label{logn4-fig}
\end{figure}

It may of course be shown analytically that, 
despite the larger number of $j$-dependent factors, one gets slower 
frequency growth in $F_{n+p}$ than in $F_n$.
Figures~\ref{logg-fig}, \ref{logn3-fig} and \ref{logn4-fig}
indicate that as the frequency growth is reduced, there will be 
increasingly larger deviations from a sinusoidal curve. 

Figure~\ref{logn4'-fig} shows the deviation from the
first term, ie the sum over $j\ge 3814281$. Notice that here the sinusoidal
structure is almost completely lost, ie the first term is 
even less dominating.

In addition to the vertical tangent at the
origin in Figure~\ref{logn4'-fig}, there seems to be approximate
self-similarities, like those for $R$ analysed by J.~Duistermaat
\cite{Dui91}. Eg the behaviour for ca.\ $40<t<75$ seems similar to that
found for $25<t<40$ and so on for $t\to0_+$.

\begin{figure}[htbp]
\epsfxsize=10cm
\epsfbox{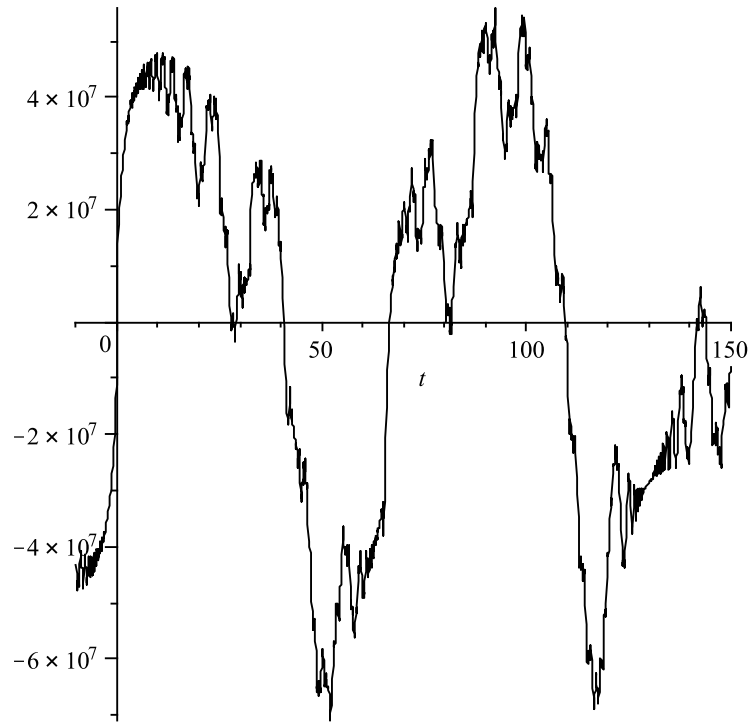}
\caption{Deviation from the first term of $\operatorname{Im} F_4(t)$ 
in Example~\ref{logn-exmp}, $n=4$, $a=2$}
  \label{logn4'-fig}
\end{figure}

\section{Final remarks}
The first example of a nowhere differentiable function is due to 
B.~Bolzano (ca.~1830, discovered 1921), 
cf the accounts in \cite{Hyk01, Thi03}. 
Nowhere-differentiability was established by means of infinite products
in \cite{Wen02}.
For a review of the historical development of the subject 
the reader could consult the illustrated
thesis of J.~Thim \cite{Thi03}.
Very recently nowhere-differentiable functions were shown to enter the
counter-examples that establish the pathological properties of
pseudo-differential operators of type~$1,1$; cf \cite{JJ08vfm}.

It has turned out that
some elements of the arguments exist sporadically in the literature; cf
Remark~\ref{ChUb-rem} for comments on \cite{ChUb07}. 
In particular $\operatorname{Re} f_{\theta}$ and $\operatorname{Im} f_{\theta}$, $\theta=1$ have
been analysed by Y.~Meyer \cite[Ch.~9.2]{Mey93} 
with a method partly based on wavelets and partly 
similar to the proof of Proposition~\ref{ftheta-prop}.
The method was attributed to G.~Freud but without any references.

However, G.~Freud showed in
\cite{Fre62} that an integrable
periodic function $f$ with Fourier series 
$\sum \rho_k\sin(n_k t+\varphi_k)$, $\inf n_{k+1}/n_k>1$ is differentiable 
at a point only if 
$\lim \rho_k n_k=0$, similarly to Theorem~\ref{exp-thm}. His proof was based
on estimates of the differentiated Cesaro means 
and of the corresponding Fej\'er kernel 
(as done also in \cite{ShSt03}), so it applies only 
to periodic functions.

Whereas the purpose in \cite[Ch.~9.2]{Mey93} was to derive
the lack of differentiability of $\operatorname{Re} f_1$, $\operatorname{Im} f_ 1$ with 
wavelet theory, the present paper goes much beyond this.
Eg nowhere-differentiability of $f_\theta$, or $W$, 
is shown to follow directly 
from basic facts in integration theory; cf the introduction. 
And using only $\mathcal{F}$, differentiability was in Theorem~\ref{exp-thm} linked
to the growth of the frequencies $b_j$.
Finally, the removal of the condition $\liminf b_{j+1}/b_j>1$ 
in Theorem~\ref{diff-thm} seems to be a novelty, which yields that the growth
of the frequency \emph{increments} $\Delta b_j$ is equally important.

\subsection*{Acknowledgement} I am grateful to Professor L.~Rodino and
Professor H.~Cornean for asking me to
publish this work; and to an anonymous referee for pointing out 
the reference \cite{BaDu92}.

%
%

\end{document}